\documentclass[10pt]{article}%{amsart}
\usepackage{amstext, amsmath, amsthm, amssymb}
\usepackage{amsmath}
\usepackage{geometry} % see geometry.pdf on how to lay out the page. There's lots.
\usepackage{color}

\numberwithin{equation}{section}
\usepackage[nottoc]{tocbibind}
\usepackage{hyperref}
\usepackage{makeidx}
\makeindex

\newtheorem{thm}{Theorem}[section]
\newtheorem{cor}[thm]{Corollary}
\newtheorem{lem}[thm]{Lemma}
\newtheorem{prop}[thm]{Proposition}
\newtheorem{defin}[thm]{Definition}
\newtheorem{rem}[thm]{Remark}
\newtheorem{example}[thm]{Example}

\newtheorem{hypoth}[thm]{Hypothesis}

\usepackage{colortbl}

\newcommand{\Z}{\mathbb{Z}}

\def\prodl{\prod\limits}

\def\suml{\sum\limits}

\def\det{\operatorname{det}}

\title{Stretching maps for tensors}

\author{Vyacheslav Futorny, Mikhail Neklyudov, Kaiming Zhao}

\AtEndDocument{\bigskip{\footnotesize%
 (V.\,Futorny) \textsc{SUSTech, Shenzhen, China} and \textsc{University of S\~ao Paulo, Brasil} \par
  \textit{E-mail address}: \texttt{vfutorny@gmail.com}  \par
   \addvspace{\medskipamount}
  (M.\,Neklyudov) \textsc{Instituto de Ci\^{e}ncias Exatas, Departamento de Matematica, UFAM, Manaus, CEP 69077-000, Brasil} \par
  \textit{E-mail address}: \texttt{misha.neklyudov@gmail.com}\par
   \addvspace{\medskipamount}
   (K.\,Zhao) \textsc{Wilfrid Laurier University, Waterloo, N2L3C5 Canada}
\par 
\textit{E-mail address}: \texttt{kzhao@wlu.ca}
}}

\bigskip

\date{}

\begin{document}
\maketitle

\begin{abstract}
We consider an algebra of even-order square tensors and introduce a stretching map (Definition \ref{def:StretchingMap}) which allows us to represent tensors as matrices. The stretching map could be understood as a generalized matricization. It conserves algebraic properties of the tensors (Theorem \ref{thm:ProductProp}). In the same time, we don't necessarily assume injectivity of the stretching map. Dropping the injectivity condition allows us to construct examples (section \ref{sec:StMapsRectSets}) of stretching maps with additional symmetry properties.
Furthermore, the noninjectivity leads to the averaging of the tensor and possibly could be used to compress the data.

%Keywords: Tensor algebras, Reshaping of Tensor, Matricization
\end{abstract}

%\begin{keyword}
%       Tensor algebras \sep Reshaping of Tensor \sep %Matricization
%        \MSC[2010] 17B05
%\end{keyword}

\section{Introduction}

The problems related to the classification, decomposition and realization of tensors appear in a multitude of fields of pure and applied science such as algebra and linear algebra (\cite{FGS2019}, \cite{BMW2017}, \cite{LMV2000},\cite{LMV2004},\cite{Kolda2001}, \cite{ZG2001}), quantum physics (\cite{WGE2016}), numerical analysis (\cite{BM2005},\cite{HK2007},\cite{HKT2005}), signal processing (\cite{Comon2001},\cite{LV2004},\cite{MB2005}), chemometrics (\cite{SBG2004}), data mining (\cite{LZYC2005}, \cite{SE2007}, \cite{STF2006}) etc. 

%different areas of applied  of mathematics, physics and complexity theory (see \cite{FGS2019}, \cite{WGE2016},\cite{BMW2017} and references therein). One of the very active research areas is a decomposition of tensors (see  \cite{KoldaBader2009} and references therein). 

A primary step  of the analysis of tensors in many algorithms is their matricization (also called reshaping or unfolding of tensors \cite{KoldaBader2009}) which transforms a given tensor into the matrix. Afterwards, it is possible to use the whole machinery of linear algebra. The aim of this paper is a better understanding of the procedure and its generalization. We consider a matrix tensor algebra and introduce a stretching map (Definition \ref{def:StretchingMap}) which allows us to represent tensors as matrices. The stretching map could be understood as a generalized matricization. It conserves algebraic properties of the tensors (Theorem \ref{thm:ProductProp}). In the same time, we don't necessarily assume injectivity of the stretching map. Dropping the injectivity condition allows us to construct examples (section \ref{sec:StMapsRectSets}) of stretching maps with additional symmetry properties.
Furthermore, the noninjectivity leads to the averaging of the tensor and possibly could be used to compress the data.
%Furthermore, it is shown that every noninjective stretching map is a composition of an injective one and averaging map.

In this paper we consider the associative algebra of even-order square tensors, that is tensors of type $(l,l), l\in\mathbb{N}$. The algebra is endowed with a product operation (see definition \ref{def:product}) which generalizes the Einstein product (\cite{LZZ2019}). This generalization allows us to investigate certain class of representations which could be non-faithful. We show that these representations can be presented as compositions of averaging operators with the map similar  to the tensor product map (Theorem \ref{thm:UnitaryEquivalence}). The averaging operator (see definition \ref{def:AveragingOperator}) destroys some  algebraic properties of the representation, for instance, its invariance with respect to the unitary transformations. In the same time, examples \ref{ex:ex_a}, \ref{ex:ex_b}, \ref{ex:ex_d} considered in section \ref{sec:StMapsRectSets}, show that a nontrivial averaging operator leads to some additional symmetries of the representation.

We first introduce some notations. Denote by ${\rm Mat}(n)$ the space of all $n\times n$ ($n\in\mathbb{N}\cup\{\infty\}$) matrices. Let $l\in\mathbb{N}$ and $\mathbb{A}\subseteq \mathbb{Z}^l$. Denote by ${\rm Mat}(\mathbb{A})$ the space of  matrices (possibly infinite) whose rows and columns are indexed by  $\mathbb{A}$.

The present paper is organized as follows. In Section 2, for any function $F:\mathbb{A}\to \mathbb{Z}$ we define the stretching map $\rho_F:{\rm Mat}(\mathbb{A})\to {\rm Mat}(\infty)$, and  the permutation operator $R_\sigma:{\rm Mat}(\mathbb{\infty})\to {\rm Mat}(\infty)$. In Section 3,  for any function $F:\mathbb{A}\to \mathbb{Z}$  we define the convolution operation on ${\rm Mat}(\mathbb{A})$ to make it into an associative algebra, and obtain some properties for the convolution product. In Section 4,   using stretching maps we establish an algorithm to compute the canonical form of a tensor matrix  $ C_1\otimes C_2 \otimes \ldots \otimes C_n$ under similarity.  In Section 5, we prove that, when $\mathbb{A}$ is finite and rectangular,   then for any   map $F:\mathbb{A}\to \mathbb{Z}$,
 the representation $\rho_F^{\mathbb{A}}$ for Mat$(\mathbb{A})$ is a composition of the averaging map $\Psi$ and the map unitary equivalent to the  tensor product map $\rho_{F}$. If $\mathbb{A}=\mathbb{Z}^l$ then the representation $\rho_F$ is a composition of the averaging map $\Psi$ and the map unitary equivalent to the universal  tensor product map $\rho_{F_{en}}$ \eqref{def:UTP}. Also we 
provide several examples  of stretching map with rectangular set $\mathbb{A}$ and non-injective $F$.

%Brualdi\\
%Rota
Authors hope that the class of representations  defined in this paper will have applications in coding theory and image processing.  In particular, the averaging operator could be of use to compress the image according to some specifications.

Throughout this paper we denote by $\mathbb{Z}$,   $\mathbb{N}$ and
$\mathbb{C}$ the set of  all integers,  
positive integers and complex numbers, respectively.  All matrices and vector spaces are assumed to be over $\mathbb{C}$.

\section{Stretching maps}
%All considered matrices have complex entries.  
Denote by ${\rm Mat}(n)$ the associative algebra of $n\times n$ ($n\in\mathbb{N}\cup\{\infty\}$) matrices   with the standard basis of matrix units $E_{ij}^{n}$, $0\leq i,j\leq n-1$. We will omit $n$ and simply write $E_{ij}$  if $n=\infty$. Denote by  $e_i^n,0\leq i\leq n-1$ the elements of the standard basis of $\mathbb{C}^n$.

Let $l\in\mathbb{N}$ and $\mathbb{A}\subset \mathbb{Z}^l$. Denote by ${\rm Mat}(\mathbb{A})$ the space of  matrices (possibly infinite) whose rows and columns are indexed by the set $\mathbb{A}$, such matrices we identify with the operators on $\mathbb C^{|\mathbb{A}|}$.

Let $T\in {\rm Mat}(\mathbb{A})$. For each   $\overline{i}, \overline{j}\in \mathbb{A}$ denote 
 by $T_{\overline{i},\overline{j}}=T_{i_1,j_1,\ldots,i_l,j_l}$ the  entry of $T$ in the $\overline{i}$th row and $\overline{j}$th column. 
 %{\color{red} Then
%$T=(T_{\overline{i},\overline{j}})_{\overline{i},\overline{j}\in \mathbb{A}}$ can be viewed as an $(l,l)$-tensor.?????}

Fix a function $F:\mathbb{A}\to \mathbb{Z}$.
%We will assume that $\mathbb{A}=\prodl_{i=1}^l\mathbb{A}_i$ %where $\mathbb{A}_1,\ldots,\mathbb{A}_l\subset \mathbb{Z}$. 
 We say that $\overline{i}\in \mathbb{A}$ is equivalent to $\overline{j}\in \mathbb{A}$ and denote $\overline{i}\backsim_F \overline{j}$ (for short, $\overline{i}\backsim \overline{j}$) if
\[
F(\overline{i})=F(\overline{j}).
\]

In particular, let  
$\overline{i}=(i_1,\ldots, i_l)$, $\overline{j}=(j_1,\ldots, j_l)$ and 
 $\overline{k}=(k_1,\ldots, k_l)$. 
 Denote by $\overline{i}\cdot \overline{j}:=\sum\limits_{s=1}^l i_s j_s$ the usual dot product in $\mathbb{C}^l$ and by $F_{\overline{k}}$ the function 
 $$F_{\overline{k}}(\overline{i})=\overline{k}\cdot \overline{i}.$$
Then $\overline{i}\backsim_{F_k} \overline{j}$ if $\overline{i} - \overline{j}$ is orthogonal to $\overline{k}$.

Denote by $\mathbb{C}^{\mathbb{A}}$ the space of complex tuples indexed by $\mathbb{A}$ with standard basis $\{e_{\overline{i}}: \overline{i}\in \mathbb{A}\}$.

\begin{defin}\label{def:StretchingMap} For any $T\in {\rm Mat}(\mathbb{A})$, we  define the matrix $\rho_F(T)\in  {\rm Mat}(\infty)$ as follows
\[
\rho_F(T):=\suml_{\overline{i},\overline{j}\in \mathbb{A}}T_{i_1,j_1,\ldots,i_l,j_l} E_{F(\overline{i}),F(\overline{j})}.
\]
We call $\rho_F$ as 
the {\bf stretching map}  associated to the function $F$.
Similarly, for any $x\in \mathbb{C}^{\mathbb{A}}$ we define the vector $\rho_F^0(x)\in \mathbb{C}^{\infty}$
as follows:
\[
\rho_F^0(x)=\sum\limits_{\overline{i}\in \mathbb{A}}x_{i_1,i_2,\ldots,i_l}e_{F(\overline{i})}.
\]
 We call $\rho_F^0$ as the {\bf vector  stretching map}.
Furthermore, we define the action $*$ of  ${\rm Mat}(\mathbb{A})$ on  $\mathbb{C}^{\mathbb{A}}$ as follows: for $T\in {\rm Mat}(\mathbb{A})$, $x\in \mathbb{C}^{\mathbb{A}}$ and for every $\overline{i}\in \mathbb{A}$ we have
\[
(T\ast x)_{\overline{i}}:=\sum\limits_{\overline{j}, \overline{l}: \overline{j}\backsim\overline{l}} T_{\overline{i},\overline{j}}x_{\overline{l}},
\]
\end{defin}
We may need some restrains on $F$ or on $T$ for $\rho_F(T)$ and $\rho_F^0(x)$ to be well-defined. In this paper when $\rho_F(T)$ and $\rho_F^0(x)$ appear we mean they are well-defined.
\begin{rem}
If $\mathbb{A}$ is a finite set and function $F$ is injective then stretching map $\rho_F$ is a generalization of notion of matricization (also called reshaping, unfolding or flattening, pp. 459-460 \cite{KoldaBader2009}) of even order square tensor. Note that we allow arbitrary permutation of elements of tensor (by the choice of $F$) comparing to the  definition of matricization in \cite{Kolda2006}, p. 10.
\end{rem}
Let $S_l$ denote the symmetric group. Denote for $\sigma\in S_l$ and $\overline{i}=(i_1,\ldots,i_l)\in\mathbb{Z}^l$, $\sigma(\overline{i}):=(i_{\sigma(1)},\ldots,i_{\sigma(l)})$.
\begin{defin} 
For each $\sigma\in S_l$ such that $\sigma(\mathbb{A})\subset \mathbb{A}$ define the permutation operator $R_{\sigma}: \rho_F({\rm Mat}(\mathbb{A}))\to \rho_F({\rm Mat}(\mathbb{A}))$ as follows
\[
R_{\sigma}(\rho_F(T)):=\rho_{F\circ\sigma}(T).
\]

In particular case when $\sigma=r$ is the reversal permutation, i.e. $r(\{1,2,\ldots,l\}):=\{l,\ldots,2, 1\}$, we will call $R_r$ the reversal operator. 
\end{defin}

One can easily verify that the map $R_{\sigma}$ above is well-defined, i.e., it does not depend on different choices of $T$. If $\rho_F$ is onto, then $R_\sigma$ gives a vector space isomorphism on ${\rm Math}(\infty).$

\begin{prop}
Assume that $\sigma_1,\sigma_2\in S_l$ and  $\sigma_p(\mathbb{A})\subset \mathbb{A}$, $p=1,2$. Then
\[
R_{\sigma_1}R_{\sigma_2}=R_{\sigma_2\sigma_1}, R_{\sigma_1}^{-1}=R_{\sigma_1^{-1}}.
\]
\end{prop}

 We obtain immediately from the definition of the map $\rho_F$ the following statement.

\begin{prop}\label{thm:Injectivity}
    If the $F:\mathbb{A}\to\mathbb{Z}$ is injective then the stretching map $\rho_F$ is injective.
\end{prop}

\section{Convolution operation}

Note that ${\rm Mat(\mathbb{A})}$ is not closed in general under the usual product of matrices. Nevertheless, we can define the following convolution operation on ${\rm Mat}(\mathbb{A})$.
\begin{defin}\label{def:product}
Fix a function $F:\mathbb{A}\to \mathbb{Z}$, let $T^1,T^2\in {\rm Mat}(\mathbb{A})$. Define
\[
(T^1\ast T^2)_{\overline{i},\overline{j}}:=\sum\limits_{\overline{m}\backsim\overline{n}} T^1_{\overline{i},\overline{m}}T^2_{\overline{n},\overline{j}}
\]
\end{defin}
We may need some restrains on $T_1$ and $T_2$ to make $T^1\ast T^2$ to be well-defined if $\mathbb{A}$ is infinite. From now on we assume that $T^1\ast T^2$ is   well-defined when we see such notations.

We have the following result which  shows how  the stretching map $\rho_F$ behaves with respect to the convolution.

\begin{thm}\label{thm:ProductProp} Let  $F:\mathbb{A}\to \mathbb{Z}$ be any function, $\mathbb{A}\subset \mathbb{Z}^l$, 
 $T^1, T^2\in {\rm Mat}(\mathbb{A})$. Then we have 
\[
\rho_F(T^1\ast T^2)=\rho_F(T^1)\rho_F(T^2).
\]
Moreover, for any $x\in \mathbb{C}^{\mathbb{A}}$ and $T\in {\rm Mat}(\mathbb{A})$ we have
\[
\rho_F^0(T \ast x)=\rho_F(T)\rho_F^0(x).
\]
\end{thm}

\begin{proof}
First equality immediately follows from product rules for the standard basis $E_{ij}$. Indeed, we have
\begin{eqnarray}
\rho_F(T^1)\rho_F(T^2) &=\sum\limits_{\overline{n}_1,\overline{n}_2} T^1_{\overline{n}_1,\overline{n}_2} E_{F(\overline{n}_1),F(\overline{n}_2)} \sum\limits_{\overline{n}_3,\overline{n}_4}T^2_{\overline{n}_3,\overline{n}_4} E_{F(\overline{n}_3),F(\overline{n}_4)}\nonumber\\
&=\sum\limits_{\overline{n}_1,\overline{n}_4} \,\,
\sum\limits_{\overline{n}_2,\overline{n}_3:F(\overline{n}_2)=F(\overline{n}_3)}T^1_{\overline{n}_1,\overline{n}_2}T^2_{\overline{n}_3,\overline{n}_4}E_{F(\overline{n}_1),F(\overline{n}_4)}\nonumber\\
&=\rho_F(T^1\ast T^2).
\end{eqnarray}
For the second equality we similarly have:

\begin{eqnarray}
\rho_F(T)\rho_F^0(x) &=\sum\limits_{\overline{n}_1}\sum\limits_{\overline{n}_2} T_{\overline{n}_1,\overline{n}_2} E_{F(\overline{n}_1),F(\overline{n}_2)} \sum\limits_{\overline{l}}x_{\overline{l}}e_{F(\overline{l})}\nonumber\\
&=\sum\limits_{\overline{n}_1}\left[
\sum\limits_{\overline{l}\backsim\overline{n}_2}T_{\overline{n}_1,\overline{n}_2}x_{\overline{l}}\right]e_{F(\overline{n}_1)}=\rho_F(Tx).
\end{eqnarray}

\end{proof}

Define ${\rm Id}\in {\rm Mat(\mathbb{A})}$ as follows:
\[
{\rm Id}_{\overline{i},\overline{j}}:=\prod\limits_{m=1}^l \delta_{i_m,j_m}.\]

\begin{prop}[Properties of convolution product $\ast$]

\begin{itemize}
\item[(a)] The convolution operation $\ast$ is associative.
\item[(b)] 
\[
(T\ast {\rm Id})_{\overline{i},\overline{j}}=\sum\limits_{\overline{m}\backsim\overline{j}}T_{\overline{i},\overline{m}},\,\,\,\,
({\rm Id} \ast T)_{\overline{i},\overline{j}}
=\sum\limits_{\overline{m}\backsim\overline{i}}T_{\overline{m},\overline{j}}.
\]
\end{itemize}
\end{prop}

\begin{proof}
\begin{trivlist}
a) Indeed, we have
\[
\left[T^1\ast(T^2\ast T^3)\right]_{\overline{i},\overline{j}}=\sum\limits_{\overline{k}\backsim\overline{l}}\sum\limits_{\overline{m}\backsim\overline{n}}T^1_{\overline{i},\overline{l}}T^2_{\overline{k},\overline{m}}T^3_{\overline{n},\overline{j}}=\left[(T^1\ast T^2)\ast T^3\right]_{\overline{i},\overline{j}}.
\]
b) Immediately follows from the definition of $\ast$.
\end{trivlist}
\end{proof}
\begin{defin}
Define the adjoint operator   $\star: {\rm Mat}(\mathbb{A})\mapsto  {\rm Mat}(\mathbb{A})$ by
\[
(\star T)_{\overline{i},\overline{j}}:= T_{\overline{j},\overline{i}}.
\]
\end{defin}

We have the following property of the adjoint operator.

\begin{lem}\label{lem:Adjoint} For $T^1,T^2\in  {\rm Mat}(\mathbb{A})$ we have
\[
\left[\rho_F(T^2\ast T^1)\right]^{{\mathrm{t}}}=\rho_F((\star T^1)\ast (\star T^2)),
\]
where $^{{\mathrm{t}}}$ means the usual transpose of the matrix.
\end{lem}
\begin{proof}
Direct calculation.
\end{proof}
%\begin{example}
%In the situation of example \ref{ex:ex_1} Lemma %\ref{lem:Adjoint} gives us
%\[
%\rho^*(T)=\rho(\star T).
%\]
%\end{example}
Suppose now that the set $\mathbb{A}$ is finite. If $T\in  {\rm Mat}(\mathbb{A})$ then $\rho_F(T)$ is an infinite matrix. 
Let  $\rho_F^{\mathbb{A}}(T)$ be the submatrix of  $\rho_F(T)$ with the rows and columns  corresponding to the set $F(\mathbb{A})$. Hence the matrix $\rho_F^{\mathbb{A}}(T)$ is a finite square matrix for any $T\in  {\rm Mat}(\mathbb{A})$.

\begin{cor}\label{cor:Hyperdet}
Assume that $\mathbb{A}$ is finite. Define the map $\kappa: {\rm Mat}(\mathbb{A})\to\mathbb{C}$ such that $\kappa(T)=\det \rho_F^{\mathbb{A}}(T)$ for $T\in {\rm Mat}(\mathbb{A})$. Then
\begin{equation}\label{eqn:Hyperdet}
\kappa(T^1\ast T^2)=\kappa(T^1)\kappa(T^2), \,\, T^1,T^2\in {\rm Mat}(\mathbb{A}).
\end{equation}
\end{cor}

\medskip

\begin{example}[Tensor product map]\label{ex:ex_1}
Suppose that the set $\mathbb{A}$ has a ``rectangular" shape, that is 
$$\mathbb{A}=\prodl_{i=1}^l\mathbb{A}_i,$$ 
where $\mathbb{A}_i=\{0,1, \ldots, n_i-1\}$, $i=1, \ldots, l$. Fix $\overline{k}\in \mathbb{A}$ such that 
$\overline{k}=(k_1, k_2,\ldots, k_l)$ with $k_1=1,k_2=n_1$, $k_3=n_1n_2$, $\ldots$, $k_{l}=n_1 n_2\ldots n_{l-1}$
and  consider the function $F_{\overline{k}}$:  $F_{\overline{k}}(\overline{i})=\overline{k}\cdot\overline{i}, \,\, \overline{i}\in \mathbb{A}$.
Then $\overline{m}\backsim\overline{n}$ is equivalent to the equality $\overline{m}=\overline{n}$ by the uniqueness of division with remainder. Consequently, this case corresponds to the tensor product map which will be considered in more detail in the next section.
%\ref{sec:TensorProduct}.
%we have
%\[
%\rho(T)=\suml_{i_1,j_1,i_2,j_2\in \mathbb{A}}T_{i_1,j_1,i_2,j_2} E_{k_1 i_1+i_2,k_1 j_1+j_2}^n
%\]
%\[
%(T^1\ast T^2)_{i_1,j_1,i_2,j_2}:=\sum\limits_{\overline{m}}
%T^1_{i_1,m_1,i_2,m_2} T^2_{m_1,j_1,m_2,j_2}
%\]
%In particular, if $T^1=A\otimes B$, $T^2=C\otimes D$, $A,B,C,D\in Mat(n)$ then
%\[
%T^1\ast T^2=AC\otimes BD.
%\]
%Thus, by proposition \ref{prop:ProductProp} restriction of $\rho$ on tensor products %$\phi(A,B):=\rho(A\otimes B)$ will be algebraic homomorphism w.r.t. both arguments i.e.
%\[
%\phi(AC,BD)=\phi(A,B)\phi(C,D).
%\]

\end{example}

If $\overline{k}$ satisfies the conditions in Example \ref{ex:ex_1}, then we denote the map $F_{\overline{k}}$ 
by $F_{\rm{TP}}$. Note that $F_{\rm{TP}}(\mathbb{A})=\{0, \ldots, n_1n_2\ldots n_l -1 \}$.

\begin{rem}
If each equivalence class is a singleton for the relation $\backsim$, as in the case of Example \ref{ex:ex_1}, then the element $\rm{Id}$ is an identity for the product $\ast$.
\end{rem}

\medskip

\section{Tensor product map}\label{sec:TensorProduct}

In this section we consider the tensor product map. Let $n_1, \ldots, n_l$ be positive integers, and 
$\mathbb{A}$ the corresponding rectangular set: 
 $$\mathbb{A}=\prodl_{i=1}^l \{0, \ldots, n_i-1\}\subset \Z^l.$$

For arbitrary matrices $A_k=(A^k_{ij})\in {\rm Mat}(n_k)$, $0\leq i,j\leq n_k-1$, $k=1, \ldots, l$, consider an $l$-tensor 
$T=A_1\otimes A_2\otimes\ldots\otimes A_l\in \bigotimes\limits_{s=1}^l {\rm Mat}(n_s)$. Then $T$ can be viewed as a matrix in Mat$(\mathbb{A})$:
\[
T_{\mathbf{i},\mathbf{j}}=T_{i_1,j_1,i_2,j_2,\ldots,i_l,j_l}=A^1_{i_1,j_1}A^2_{i_2,j_2}\ldots A^l_{i_l,j_l},
\]
where $\mathbf{i}=(i_1,\ldots,i_l)$, $\mathbf{j}=(j_1,\ldots,j_l)$, $0\leq i_s,j_s\leq n_s-1$, $s=1,\ldots,l$.

Then the \emph{tensor product} map $$\rho_{n_1, \ldots, n_l}: \bigotimes\limits_{s=1}^l {\rm Mat}(n_s) \rightarrow {\rm Mat}(n_1n_2\ldots n_l)$$ is defined explicitly as follows: 

\[
\rho_{n_1, \ldots, n_l}(T)=\sum\limits_{i_1,j_1=0}^{n_1-1}\sum\limits_{i_2,j_2=0}^{n_2-1}\ldots\sum\limits_{i_l,j_l=0}^{n_l-1}T_{i_1,j_1,i_2,j_2,\ldots,i_l,j_l}E^{N}_{I,J},
\]
where $N=n_1n_2\ldots n_l$, $I=i_1+n_1 i_2+n_1 n_2 i_3+\ldots +n_1\ldots n_{l-1} i_l$, $J=j_1+n_1 j_2+n_1 n_2 j_3+\ldots +n_1\ldots n_{l-1} j_l$.

Consider 
 the   stretching map $\rho_{\rm{TP}}=\rho_{F_{\rm{TP}}}$ and the induced map 
 $$\rho_{\rm{TP}}^f: {\rm Mat}(\mathbb{A})\rightarrow {\rm Mat} (n_1 \ldots n_l).$$
 Identifying an $l$-tensor 
$T$ with the matrix $(T_{\mathbf{i},\mathbf{j}})_{\mathbf{i},\mathbf{j}\in \mathbb{A}}$, we can view $\bigotimes\limits_{s=1}^l {\rm Mat}(n_s)$ as a subspace of
${\rm Mat}(\mathbb{A})$. 
We have $\rho_{\rm{TP}}^f(T)=\rho_{n_1, \ldots, n_l}(T)$. Hence,
 the tensor product map $\rho_{n_1, \ldots, n_l}$ is induced by the restriction of $\rho_{\rm{TP}}^f$ on $\bigotimes\limits_{s=1}^l {\rm Mat}(n_s)$.

%We will also identify a tensor $T\in \bigotimes\limits_{s=1}^l Mat(n_s)$ with the linear operator $T\in L(\bigotimes\limits_{s=1}^l \mathbb{C}^{n_s})$ via the following standard action:
%\[
%(Tx)_{i_1,i_2,\ldots,i_l}=\sum\limits_{j_1,j_2,\ldots,j_l}T_{i_1,j_1,i_2,j_2,\ldots,i_l,j_l} x_{j_1,j_2,\ldots,j_l},\,\, x\in \bigotimes\limits_{s=1}^l \mathbb{C}^{n_s}. 
%\]

Similarly, we have the following local form of the vector stretching  map.
Let $x\in \bigotimes\limits_{s=1}^l \mathbb{C}^{n_s}$ be an $l$-tensor. If $\mathbf{i}=(i_1,\ldots,i_l)$, then we set 
$x_{\mathbf{i}}=\{x_{i_1,i_2,\ldots,i_l}\}$, $0\leq i_s\leq n_s-1$, $s=1,\ldots,l$. Then the correspondence $x\mapsto (x_{\mathbf{i}})$  
defines the
 map $$\rho^0_{n_1, \ldots, n_l}:\bigotimes\limits_{s=1}^l\mathbb{C}^{n_s}\to \mathbb{C}^{n_1n_2\ldots n_l},$$ which can be written explicitly as follows:
\[
\rho^0_{n_1, \ldots, n_l}(x)=\sum\limits_{i_1=0}^{n_1-1}\sum\limits_{i_2=0}^{n_2-1}\ldots\sum\limits_{i_l=0}^{n_l-1}x_{i_1,i_2,\ldots,i_l}e^{N}_{I},
\]
where $N=n_1n_2\ldots n_l$, $I=i_1+n_1 i_2+n_1 n_2 i_3+\ldots +n_1\ldots n_{l-1} i_l$. Identifying $\bigotimes\limits_{s=1}^l\mathbb{C}^{n_s}$ with $\mathbb{C}^{\mathbb{A}}$, we  see that 
 the  map $\rho^0_{n_1, \ldots, n_l}$ coincides with $\rho_{F_{\rm{TP}}}^0$.

The convolution operation on $\bigotimes\limits_{s=1}^l \rm{Mat}(n_s)$ will have the following form:
\[
(T^1\ast T^2)_{\mathbf{i},\mathbf{j}}=\suml_{\mathbf{m}} T^1_{\mathbf{i},\mathbf{m}} T^2_{\mathbf{m},\mathbf{j}},\,\,T^1,T^2\in \bigotimes\limits_{s=1}^l \rm{Mat}(n_s).
\]
From Theorem \ref{thm:ProductProp}  we immediately have 

\begin{cor}\label{cor:ProductRule_1} 
\item[(a)] For any $T^1,T^2\in \bigotimes\limits_{s=1}^l {\rm Mat}(n_s)$ we have
\[
\rho_{n_1, \ldots, n_l}(T^1\ast T^2)=\rho_{n_1, \ldots, n_l}(T^1)\rho_{n_1, \ldots, n_l}(T^2).
\]
\item[(b)] For any $x\in \bigotimes\limits_{s=1}^l \mathbb{C}^{n_s}$ and  $T\in \bigotimes\limits_{s=1}^l {\rm Mat}(n_s)$ we have
\[
\rho^0_{n_1, \ldots, n_l} (T\ast x)=\rho_{n_1, \ldots, n_l}(T)\rho^0_{n_1, \ldots, n_l}(x).
\]
\item[(c)] Let $\kappa:\bigotimes\limits_{s=1}^l {\rm Mat}(n_s)\to\mathbb{C}$, $\kappa(T)=\det \rho_{n_1, \ldots, n_l}(T)$. Then
\begin{equation}\label{eqn:Hyperdet_2}
\kappa(T^1\ast T^2)=\kappa(T^1)\kappa(T^2), \,\, T^1,T^2\in\bigotimes\limits_{s=1}^l {\rm Mat}(n_s).
\end{equation}
\end{cor}

%\begin{proof}
%  Note  that the equality
%\[
%i_1+n_1 i_2+n_1 n_2 i_3+\ldots +n_1\ldots n_{l-1} %i_l=j_1+n_1 j_2+n_1 n_2 j_3+\ldots +n_1\ldots n_{l-1} j_l
%\]
% for $0\leq i_s,j_s\leq n_s-1$, $s=1,\ldots,l$ implies that %$i_s=j_s$, $s=1,\ldots,l$ by sequential application of the %uniqueness of the division with the remainder. Now the proof %follows from commutation relations for matrix units.
%\end{proof}

As a consequence we have the following properties of the map $\rho_{n_1, \ldots, n_l}$.

\begin{cor}\label{cor:RhoProperties}
\begin{trivlist}
\item[(a)] Let $A_s,B_s\in {\rm Mat}(n_s)$, $s=1,\ldots,l$. Then
\[
\rho_{n_1, \ldots, n_l}(A_1 B_1\otimes A_2 B_2\otimes\ldots\otimes A_lB_l)=\rho_{n_1, \ldots, n_l}(A_1\otimes A_2\otimes\ldots\otimes A_l)\rho_{n_1, \ldots, n_l}(B_1\otimes B_2\otimes\ldots\otimes B_l).
\]\item[(b)] Let $I_k$ denote the identity matrix of size $k\times k$. Then
 \[
\rho_{n_1, \ldots, n_l}(I_{n_1}\otimes I_{n_2}\otimes\ldots\otimes I_{n_l} )=I_{n_1 n_2\ldots n_l}.
\]
\item[(c)] The map
$\rho_{n_1, \ldots, n_l}:\bigotimes\limits_{s=1}^l {\rm Mat}(n_s)\to {\rm Mat}(n_1n_2\ldots n_l)$ is an associative algebra isomorphism. 

\end{trivlist}
\end{cor}

\begin{proof}
Statements (a) and (b) follow from Corollary \ref{cor:ProductRule_1}, while (c) follows from the fact that given 
$i_1, \ldots, i_l$ such that
$0\leq i_s\leq n_s-1$, $s=1,\ldots,l$, any $n\in \mathbb{Z}$ between $0$ and $n_1n_2\ldots n_l-1$ can be uniquely represented in the form
\[
n=i_1+n_1 i_2+n_1 n_2 i_3+\ldots +n_1\ldots n_{l-1} i_l.
\]

\end{proof}

%{\red
%\begin{rem}
%How is this definition of determinant for tensor is related %to different definitions of Hyperdeterminants of tensors? %In this fashion we could define characteristic polynomial %for the tensor $T$ as $\kappa(T-\lambda Id)$. What would %look like corresponding " Jordan form"? How it will depend %upon parameter vector $(n_1,n_2,\ldots,n_l)$? We could fix %the number $N=:n_1 n_2\ldots n_l$. Will the result depend %upon which decomposition of $N$ we will choose?
%
%Notice that $\kappa(Id_{{\rm Mat}(k)}\otimes A)=\det^k A$. Can we find more general formula for $\kappa$?
%\end{rem}
%}
Let us remind the reader that the adjoint  operator on the tensor algebra $$\star:\bigotimes\limits_{s=1}^l {\rm Mat}(n_s)\mapsto \bigotimes\limits_{s=1}^l {\rm Mat}(n_s)$$  has the form:
\[
(\star T)_{\mathbf{i}, \mathbf{j}}:= T_{\mathbf{j}, \mathbf{i}},\,T\in \bigotimes\limits_{s=1}^l {\rm Mat}(n_s).
\]
Denote by $\rho^*_{n_1, \ldots, n_l}(T)$ the adjoint element of $\rho_{n_1, \ldots, n_l}(T)$ in the algebra ${\rm Mat}(n_1n_2\ldots n_l)$.
We have the following properties of the $\star$ map.

\medskip

Let  $C, D\in {\rm Mat}(n)$ for some $n\in \mathbb N\cup \{\infty\}$. Recall that these matrices are  similar and write $C\backsim D$ if there exists an ivertible  matrix $U\in {\rm Mat}(n)$ such that $D=U C U^{-1}$. We have the following corollary.

\begin{cor}\label{cor:RhoProperties-star}
\begin{trivlist}
\item[(a)] For any $T\in \bigotimes\limits_{s=1}^l {\rm Mat}(n_s)$
\[
(\rho_{n_1, \ldots, n_l}(T))^{\rm{t}}=\rho_{n_l, \ldots, n_1}(\star T).
\]
\item[(b)]
Let $C_i, D_i\in {\rm Mat}(n_i)$, $i=1,\ldots,l$ such that $C_i$ is similar to $D_i$  for all $i$. Then $\rho_{n_1, \ldots, n_l}(C_1\otimes\ldots \otimes C_l)$ and $\rho_{n_1, \ldots, n_l}(D_1\otimes\ldots \otimes D_l)$ are similar.
\end{trivlist}
\end{cor}

\begin{proof}
Statement (a) follows immediately from the definition of the $\star$ operator.
Let us show part (b). We have $D_i=U_i C_i U_i^*$, $i=1,\ldots,l$, where $U_i\in {\rm Mat}(k_i)$ is a unitary matrix. Then by part (a) and Corollary \ref{cor:RhoProperties-star},(a) we have 
\[
\rho_{n_1, \ldots, n_l}(D_1\otimes D_2\otimes\ldots\otimes D_l)=\]
\[\prod\limits_{i=1}^l\rho_{n_1, \ldots, n_l}(I\otimes\ldots\otimes U_i\otimes\ldots I)\rho_{n_1, \ldots, n_l}(C_1\otimes C_2\otimes\ldots\otimes C_l)\left[\prod\limits_{i=1}^l \rho_{n_1, \ldots, n_l}(I\otimes\ldots\otimes U_i\otimes\ldots I)\right]^{-1}.
\]
Hence, the result follows.
\end{proof}

\begin{example}
Let $l=2$. Then easy calculation demonstrates that the map $${\rm Mat}(n)\ni A\mapsto \rho_{k,n}({\rm Id}_{{\rm Mat}(k)}\otimes A)\in {\rm Mat}(kn)$$ is an algebraic homomorphism, whose image has ``stretches" of $A$'s  $k$ times along both diagonal and perpendicular diagonal.  This example justifies  the name of the stretching map.
\end{example}
We have the following statement.

\begin{prop} Let $B_i\in {\rm Mat}(n_i)$, $i=1,\ldots,l$. Then
\begin{itemize}
    \item[(a)] 
\[
R_{\sigma}(\rho_{n_1, \ldots, n_l}(B_1\otimes B_2\otimes\ldots \otimes B_l))=\rho_{n_{\sigma(1)}, \ldots, n_{\sigma(l)}}(B_{\sigma(1)}\otimes B_{\sigma(2)}\otimes\ldots \otimes B_{\sigma(l)}).
\]
In particular,
\[
R_{r}(\rho_{n_1, \ldots, n_l}(B_1\otimes B_2\otimes\ldots \otimes B_l))=\rho_{n_{l}, \ldots, n_{1}}(B_{l}\otimes B_{l-1}\otimes\ldots \otimes B_{1}).
\]
\item[(b)] For any $\sigma\in S_l$ the operator $R_{\sigma}$ is an isometry.
\end{itemize}
\end{prop}

\begin{proof} The statement (a) is straightforward from the definitions. 
The isometry property follows from the observation that $R_{\sigma}$ permutes the  elements of the basis of ${\rm Mat}(n_1 n_2\ldots n_l)$. Indeed, as we have noticed in the Corollary \ref{cor:ProductRule_1}, any $m\in\mathbb{Z}$ such that $0\leq m\leq n_1 n_2\ldots n_l-1$ can be uniquely represented in the form $m=\overline{k}\cdot \overline{i}$, where $\overline{k}=(1,n_1,n_1n_2,\ldots,n_1 n_2\ldots n_{l-1})$, $\overline{i}=(i_1,\ldots, i_l)$, $0\leq i_s\leq n_s-1$, $s=1,\ldots, l$. Thus, the operator $R_{\sigma}$ is a bijection which sends matrix units $E_{ml}^{n_1 n_2\ldots n_l}$, $0\leq m,l\leq n_1 n_2\ldots n_l-1$ into matrix units.
\end{proof}

 \begin{rem}
 The reversal operator $R_r$ provides another way to deduce part (b) of Corollary \ref{cor:RhoProperties-star}. We restrict ourselves to the case $l=2$. General case is treated analogously. We have
 \[
\rho_{n_1, n_2}(C_1\otimes C_2)\backsim \rho_{n_1, n_2}(D_1\otimes C_2)
\]
$\Rightarrow$
\[ \rho_{n_2, n_1}(C_2\otimes C_1)\backsim \rho_{n_2,  n_1}(C_2\otimes D_1)= R_r(\rho_{n_1,   n_2}(D_1\otimes C_2))\]
\[ \backsim R_r(\rho_{n_1,   n_2}(D_1\otimes D_2))=\rho_{n_1,  n_2}(D_2\otimes D_1).
\]
Applying the reversal operator  $R_r$ again, we get the result.
\end{rem}

Part (b) of Corollary \ref{cor:RhoProperties-star} allows us to classify $\rho_{n_1, \ldots, n_l} (C_1\otimes\ldots \otimes C_l)$ up to similarity. Indeed, it is enough to compute $\rho_{n_1, \ldots, n_l}$ of tensor products of direct sums of Jordan blocks. Here we will calculate $\rho_{n_1, \ldots, n_l}$ of the tensor product of two sums of Jordan blocks. The general case is more cumbersome and can be obtained in a similar fashion. We will use the following formula

\begin{lem}\label{lem:SumTensor} Let $C_i\in\rm{Mat}(\mu_i)$, $D_i\in\rm{Mat}(\nu_i)$ for $i=1,2$. 

(a). $$\rho_{{\mu}_1, {\nu}_1}(C_1\otimes D_1 )\backsim \rho_{{\nu}_1, {\mu}_1}(D_1\otimes C_1 ).$$

(b). 
\[\rho_{{\mu}_1+{\mu}_2, {\nu}_1+{\nu}_2}((C_1\oplus C_2)\otimes(D_1\oplus D_2))\]
\[\backsim \rho_{{\mu}_1, {\nu}_1}(C_1\otimes D_1 )\oplus\rho_{{\mu}_1, {\nu}_2}( C_1\otimes D_2) \oplus \rho_{{\mu}_2, {\nu}_1}(C_2\otimes D_1) \oplus \rho_{{\mu}_2, {\nu}_2}(C_2\otimes D_2).
\]
\end{lem}
\begin{proof}
Reorder the rows and the columns in the same pattern at the same time.
\end{proof}

We have also the following formula for the image of $\rho_{n_1, \ldots, n_l}$ up to similarity.

\begin{thm}\label{thm:JF} Denote by $J_{\mu}(a)$ the Jordan cell of size $\mu$ with eigenvalue $a$.
Let 
\[
C:=\bigoplus\limits_{i=1}^m J_{\mu_i}(a_i), \, D:=\bigoplus\limits_{j=1}^n J_{\nu_j}(b_j),
\]
where
\[
\overline{\mu}_l:=\sum\limits_{k=0}^l\mu_k, \,\,\overline{\nu}_s:=\sum\limits_{k=0}^s\nu_k,l=0,\ldots,m; \,\, s=0,\ldots,n\, (\mu_0=\nu_0=0).
\]
Then $\rho_{\overline{\mu}_m, \overline{\nu}_n}(C\otimes D) =$
\begin{eqnarray}
= \sum\limits_{l,s=0}^{m,n} a_lb_s\sum\limits_{i=\overline{\mu}_l}^{\overline{\mu}_{l+1}-1} 
\sum\limits_{j=\overline{\nu}_s}^{\overline{\nu}_{s+1}-1} E_{i+\overline{\mu}_m j, i+\overline{\mu}_m j} +
 \sum\limits_{l,s=0}^{m,n} b_s\sum\limits_{i=\overline{\mu}_l}^{\overline{\mu}_{l+1}-2} 
\sum\limits_{j=\overline{\nu}_s}^{\overline{\nu}_{s+1}-1} E_{i+\overline{\mu}_m j, i+\overline{\mu}_m j+1}\label{eqn:Classification}\\
+ \sum\limits_{l,s=0}^{m,n} a_l\sum\limits_{i=\overline{\mu}_l}^{\overline{\mu}_{l+1}-1} 
\sum\limits_{j=\overline{\nu}_s}^{\overline{\nu}_{s+1}-2} E_{i+\overline{\mu}_m j, i+\overline{\mu}_m j+\overline{\mu}_m}
+ \sum\limits_{l,s=0}^{m,n} \sum\limits_{i=\overline{\mu}_l}^{\overline{\mu}_{l+1}-2} 
\sum\limits_{j=\overline{\nu}_s}^{\overline{\nu}_{s+1}-2} E_{i+\overline{\mu}_m j, i+\overline{\mu}_m j+\overline{\mu}_m+1}.\nonumber
\end{eqnarray}
\end{thm}
\begin{rem}
Theorem \ref{thm:JF} could be used to classify the image of $\rho_{n_1, \ldots, n_l}$ up to unitary equivalence. 
\end{rem}
\begin{rem}
Notice that the theorem above shows that $$\rho_{\mu, \nu}(J_{\mu}(a)\otimes J_{\nu}(b))\neq \rho_{\nu.\mu}(J_{\nu}(b)\otimes J_{\mu}(a))$$ even in the case $\mu=\nu>1$ and $a\neq b$. Indeed, we have
\begin{eqnarray}
\rho_{\mu, \nu}(J_{\mu}(a)\otimes J_{\nu}(b)) &= ab\suml_{i=0}^{\mu-1}\suml_{k=0}^{\nu-1} E_{i+\mu k,i+\mu k} + b\suml_{i=0}^{\mu-2}\suml_{k=0}^{\nu-1} E_{i+\mu k,i+\mu k+1}\nonumber\\
&+a\suml_{i=0}^{\mu-1}\suml_{k=0}^{\nu-2} E_{i+\mu k,i+\mu k+\mu} +\suml_{i=0}^{\mu-2}\suml_{k=0}^{\nu-2} E_{i+\mu k,i+\mu k+\mu+1}. \nonumber
\end{eqnarray}    
\end{rem}

  We have the following result.

\begin{thm} Let $p, q\in \mathbb{N}, a,b\in\mathbb{C}.$

\medskip

(1) When $ab \neq 0$,
$$\rho_{p,q}(J_p(a)\otimes J_q(b) )\backsim \oplus_{k=1}^{\min\{p, q\}}J_{p+q-2k+1}(ab).
$$

(2) When $a \neq 0$,
$$\rho_{p,q}(J_p(a)\otimes J_q(0) )\backsim \oplus_{k=1}^{p}J_{q}(a).
$$

(3) When $b \neq 0$,
$$\rho_{p,q}(J_p(0)\otimes J_q(b)) \backsim \oplus_{k=1}^{q}J_{p}(b).
$$

(4)
$$\rho_{p,q}(J_p(0)\otimes J_q(0) )\backsim \oplus_{k=1}^{\min\{p, q\}-1}(J_{k}(0)\oplus J_{k}(0))\oplus_{k=1}^{|p-q|+1}J_{\min\{p, q\}}(0).
$$
\end{thm}

\begin{proof}
The result can be deduced from 
 the Theorem of Aitken and Roth (see \cite[Theorem 4.6]{Brualdi1985}, \cite{A1934}, \cite{Roth1934}).
\end{proof}

Now we can compute the Jordan canonical form for the matrix $\rho_{\mu_1, \mu_2, \ldots, \mu_n  }(C_1\otimes C_2\otimes \ldots\otimes C_n)$ where $C_i\in\rm{Mat}(\mu_i)$ by using Lemma 4.7 and  Theorem 4.10.

\medskip

\section{Stretching maps for rectangular sets}\label{sec:StMapsRectSets}

Let $n_1, \ldots, n_l$ be positive integers. We say that a set $\mathbb{A}\subset \mathbb{Z}^l$
is \emph{rectangular} if 
$$\mathbb{A}= \prod_{s=1}^l [0,n_s-1]  \cap \mathbb{Z}^l.$$ 

\begin{thm}\label{thm:UnitaryEquivalence}
    Let $\mathbb{A}\subset \mathbb{Z}^l$ be a finite  rectangular set and
    $F:\mathbb{A}\to \mathbb{Z}$  an injective  function.
      Then the stretching map $\rho_F^{\mathbb{A}}$ is similar to the tensor product map $\rho_{\rm{TP}}^{f}$.
\end{thm}

\begin{proof}
Let  $\mathbb{A}= \prod_{s=1}^l [0,n_s-1]   \cap \mathbb{Z}^l$. 
Without loss of generality we can assume that the image of the set $\mathbb{A}$ is
$$F(\mathbb{A})=\{0, \ldots, n_1n_2\ldots n_l -1 \}=F_{\rm{TP}}(\mathbb{A}).$$

Let $F_{\rm{TP}}$ be a function from example \ref{ex:ex_1} corresponding to the tensor product. Since $\mathbb{A}$ is finite then any other injective function from $\mathbb{A}$ to $F_{\rm{TP}}(\mathbb{A})$ is a composition of a certain permutation  of the set $F_{TP}(\mathbb{A})$ with the function $F_{\rm{TP}}$. Hence $F=\sigma\circ F_{\rm{TP}}$  for some permutation $\sigma$. Consequently, we have
\[
\rho_F^{\mathbb{A}}(T)=\sum\limits_{\overline{i},\overline{j}\in\mathbb{A}}T_{\overline{i},\overline{j}} E_{\sigma(F_{\rm{TP}}(\overline{i})),\sigma(F_{\rm{TP}}(\overline{j}))}=U_{\sigma} \rho_{\rm{TP}}^{f}(T) U_{\sigma}^{-1},
\]
where $U_{\sigma}$ is a the matrix of the  transformation of rows and columns of $\rho_{F_{\rm{TP}}}$ corresponding to the permutation $\sigma$.
\end{proof}

Next we consider that case $\mathbb{A}=\mathbb{Z}^l$ and define the \emph{universal} tensor map.
Let $F_{en}:\mathbb{Z}^l\to \mathbb{Z}$ be an enumeration (bijection) of $\mathbb{Z}^l$. The corresponding universal tensor map 
$\rho_{F_{en}}$ is defined as follows: for any $T\in\mathbb{Z}^l$
\begin{equation}\label{def:UTP}
\rho_{F_{en}}(T):=\sum\limits_{\overline{i},\overline{j}\in\mathbb{Z}^l}T_{\overline{i},\overline{j}} E_{F_{en}(\overline{i}),F_{en}(\overline{j})}.
\end{equation}
\begin{rem}\label{rem:InfIndRho}
Let $\mathbb{A}=\mathbb{Z}^l$ and $F_{en}:\mathbb{Z}^l\to \mathbb{Z}$ be an enumeration of $\mathbb{Z}^l$. Then the the restriction of the universal tensor map 
$\rho_{F_{en}}$ on matrices $T$ with finite support (i.e. if $T$ has zero entries outside of some fixed finite ``rectangle")
is similar to the tensor product map. Note that globally this is not the case.
\end{rem}

\begin{rem}
The universal tensor map $\rho_{F_{en}}$ is unique up to similarity  by the same reasoning as in the proof of  Theorem \ref{thm:UnitaryEquivalence}.
\end{rem}
\begin{defin}\label{def:AveragingOperator} Fix a function $F:\mathbb{A}\to \mathbb{Z}$.
 Define the averaging map $\Psi: {\rm Mat}(\mathbb{A})\to {\rm Mat}(\mathbb{A})$ as follows
 \[
\Psi (T):={\rm{Id}}\ast (T\ast {\rm{Id}}), T\in {\rm Mat}(\mathbb{A}).
\]
\end{defin}

Note that if $F$ is injective, then $\Psi$ is the identity map.
\begin{thm}
    If $\mathbb{A}$ is finite rectangular then the stretching map $\rho_F$ is a composition of averaging map $\Psi$ and the map similar to tensor product map. If $\mathbb{A}=\mathbb{Z}^l$ then the stretching map $\rho_F$ is a composition of averaging map $\Psi$ and the universal tensor product map $\rho_{F_{en}}$.   
\end{thm}
\begin{proof}
We have
\[
\rho_F(\Psi (T))=\rho_F(T), T\in {\rm Mat}(\mathbb{A}),
\]
by Theorem \ref{thm:ProductProp}. $\rho_F$ restricted on the image $\Psi({\rm Mat}(\mathbb{A}))$ will be injective. Consequently, in the case of finite $\mathbb{A}$, it will be similar to tensor product map by Theorem \ref{thm:UnitaryEquivalence}. If $\mathbb{A}=\mathbb{Z}^l$ it will be similar to the universal tensor product map $\rho_{F_{en}}$.
\end{proof}

Now we consider some examples of stretching map with rectangular set $\mathbb{A}$ and non-injective $F$.

\begin{example}\label{ex:ex_a}
Let $\mathbb{A}=\prodl_{i=1}^2\mathbb{A}_i$,  $\mathbb{A}_1=\mathbb{A}_2=\{0,1\}$, $F(\overline{i})=\overline{k}\cdot\overline{i}$ with $\overline{k}=(1,1)$. Then 
\[
\rho_F^{\mathbb{A}}: {\rm Mat}(2)\otimes {\rm Mat}(2)\to {\rm Mat}(3), \rho_F^{\mathbb{A}}(T)=\suml_{\overline{i},\overline{j}\in \{0,1\}^2}T_{i_1,j_1,i_2,j_2} E_{i_1+i_2,j_1+j_2}.
\]
In particular, if 
\[
A=\left(
\begin{array}{cc}
  a_{00}   & a_{01} \\
  a_{10}   & a_{11}
\end{array}
\right),
B=\left(
\begin{array}{cc}
  b_{00}   & b_{01} \\
  b_{10}   & b_{11}
\end{array}
\right)
\]
then
\[
\rho_F^{\mathbb{A}}(A\otimes B)=\left(
\begin{array}{ccc}
  a_{00}b_{00}   & a_{01}b_{00}+a_{00}b_{01} & a_{01}b_{01}  \\
  a_{00}b_{10}+a_{10}b_{00}   & a_{00}b_{11}+ a_{11}b_{00}+a_{01}b_{10}+a_{10}b_{01}  & a_{01}b_{11}+a_{11}b_{01}   \\
  a_{10}b_{10}   & a_{10}b_{11}+a_{11}b_{10} & a_{11}b_{11}     
\end{array}
\right).
\]
This formula can be rewritten as follows
\[
\rho_F^{\mathbb{A}}(A\otimes B)=
b_{00}\left(
\begin{array}{cc}
  A   & 0 \\
  0   & 0
\end{array}
\right)
+b_{11}
\left(
\begin{array}{cc}
  0   & 0 \\
  0   & A
\end{array}
\right)
+b_{10}
\left(
\begin{array}{cc}
  0   & 0 \\
  A   & 0
\end{array}
\right)
+b_{01}
\left(
\begin{array}{cc}
  0   & A \\
  0   & 0
\end{array}
\right).
\]
Notice that in this case, $\rho(A\otimes B)=\rho(B\otimes A)$.
We can calculate
\[
\rho_F^{\mathbb{A}}\left(
\left(
\begin{array}{cc}
  \lambda   & 1 \\
  0   & \lambda
\end{array}
\right)
\otimes
\left(
\begin{array}{cc}
  \mu   & 1 \\
  0   & \mu
\end{array}
\right)
\right)=
\left(
\begin{array}{ccc}
  \lambda\mu   & \lambda+\mu & 1 \\
  0   & 2\lambda\mu & \lambda+\mu\\
  0   & 0 & \lambda\mu
\end{array}
\right),
\]
\[
\rho_F^{\mathbb{A}}\left(
\left(
\begin{array}{cc}
 a_{00}   & a_{01} \\
  a_{10}   & a_{11}
\end{array}
\right)
\otimes
\left(
\begin{array}{cc}
  1   & 0 \\
  0   & 1
\end{array}
\right)
\right)=
\left(
\begin{array}{ccc}
  a_{00}   & a_{01} & 0 \\
  a_{10}   & a_{00}+a_{11} & a_{01}\\
  0   & a_{10} & a_{11}
\end{array}
\right)=
\left(
\begin{array}{cc}
  A   & 0 \\
  0   & 0
\end{array}
\right)
+
\left(
\begin{array}{cc}
  0   & 0 \\
  0   & A
\end{array}
\right).
\]
Nontrivial part of equivalence relationship $\backsim$ has form  $(0,1)\backsim (1,0)$ i.e. $\overline{m}\backsim\overline{n}$ iff
$\overline{m}=\overline{n}$ or $\overline{m}=\overline{n}^{t}$ (where $(m_1,m_2)^t=(m_2,m_1)$). Consequently,
\begin{eqnarray}
 (T^1\ast T^2)_{\overline{i},\overline{j}} &=\sum\limits_{\overline{m}} T^1_{\overline{i},\overline{m}}T^2_{\overline{m},\overline{j}}+
\sum\limits_{\overline{m}\neq\overline{m}^t} T^1_{\overline{i},\overline{m}}T^2_{\overline{m}^t,\overline{j}}\nonumber\\ 
&=\sum\limits_{\overline{m}} T^1_{\overline{i},\overline{m}}T^2_{\overline{m},\overline{j}}+ 
T^1_{\overline{i},(0,1)}T^2_{(1,0),\overline{j}}+
T^1_{\overline{i},(1,0)}T^2_{(0,1),\overline{j}}.
\end{eqnarray}
In particular,
\[
\left[(A\otimes B)\ast (C\otimes D)\right]_{\overline{i},\overline{j}}=\left[AC\otimes BD\right]_{\overline{i},\overline{j}}+...
\]
\begin{rem}
It would be interesting to find an explicit formula for $(A\otimes B)\ast (C\otimes D)$ in terms of tensor products  of $A,B,C,D$.
\end{rem}
\end{example}
\begin{example}\label{ex:ex_b}
Let $\mathbb{A}=\prodl_{i=1}^2\mathbb{A}_i$,  $\mathbb{A}_1=\mathbb{A}_2=\{0,1\}$, $F(\overline{i})=\overline{k}\cdot\overline{i}$ with $\overline{k}=(1,-1)$. We will simply write $\rho_F^{\mathbb{A}}$ as $\rho.$ Then 
\[
\rho: {\rm Mat}(2)\otimes {\rm Mat}(2)\to {\rm Mat}(3), \,\, \rho(T)=\suml_{\overline{i},\overline{j}\in \{0,1\}^2}T_{i_1,j_1,i_2,j_2} E_{i_1-i_2,j_1-j_2}.
\]

In particular, if 
\[
A=\left(
\begin{array}{cc}
  a_{00}   & a_{01} \\
  a_{10}   & a_{11}
\end{array}
\right),
B=\left(
\begin{array}{cc}
  b_{00}   & b_{01} \\
  b_{10}   & b_{11}
\end{array}
\right)
\]
then
\[
\rho(A\otimes B)=\left(
\begin{array}{ccc}
  a_{00}b_{11}   & a_{00}b_{10}+a_{01}b_{11} & a_{01}b_{10}  \\
  a_{00}b_{01}+a_{10}b_{11}   & a_{00}b_{00}+ a_{11}b_{11}+a_{01}b_{01}+a_{10}b_{10}  & a_{01}b_{00}+a_{11}b_{10}   \\
  a_{10}b_{01}   & a_{10}b_{00}+a_{11}b_{01} & a_{11}b_{00}     
\end{array}
\right).
\]
Let us denote
\[
B'=\left(
\begin{array}{cc}
  b_{11}   & b_{10} \\
  b_{01}   & b_{00}
\end{array}
\right)
\]
i.e. $B$ reflected both w.r.t. main diagonal and contra-diagonal. Then formula for $\rho(A\otimes B)$ can be rewritten as follows
\[
\rho(A\otimes B)=
a_{00}\left(
\begin{array}{cc}
  B'   & 0 \\
  0   & 0
\end{array}
\right)
+a_{11}
\left(
\begin{array}{cc}
  0   & 0 \\
  0   & B'
\end{array}
\right)
+a_{10}
\left(
\begin{array}{cc}
  0   & 0 \\
  B'   & 0
\end{array}
\right)
+a_{01}
\left(
\begin{array}{cc}
  0   & B' \\
  0   & 0
\end{array}
\right).
\]
Notice that in this case, $\rho'(A\otimes B)=\rho(B\otimes A)$ (where $'$ is a reflection both w.r.t. main diagonal and contra-diagonal). In particular,
\[
\rho\left(
\left(
\begin{array}{cc}
 a_{00}   & a_{01} \\
  a_{10}   & a_{11}
\end{array}
\right)
\otimes
\left(
\begin{array}{cc}
  1   & 0 \\
  0   & 1
\end{array}
\right)
\right)=
\left(
\begin{array}{ccc}
  a_{00}   & a_{01} & 0 \\
  a_{10}   & a_{00}+a_{11} & a_{01}\\
  0   & a_{10} & a_{11}
\end{array}
\right)=
\left(
\begin{array}{cc}
  A   & 0 \\
  0   & 0
\end{array}
\right)
+
\left(
\begin{array}{cc}
  0   & 0 \\
  0   & A
\end{array}
\right)
\]
In the same time,
\[
\rho\left(
\left(
\begin{array}{cc}
 1   & 0 \\
  0   & 1
\end{array}
\right)
\otimes
\left(
\begin{array}{cc}
 a_{00}   & a_{01} \\
  a_{10}   & a_{11}
\end{array}
\right)
\right)=
\left(
\begin{array}{cc}
  A'   & 0 \\
  0   & 0
\end{array}
\right)
+
\left(
\begin{array}{cc}
  0   & 0 \\
  0   & A'
\end{array}
\right)
\]
\[
\rho\left(
\left(
\begin{array}{cc}
  \lambda   & 1 \\
  0   & \lambda
\end{array}
\right)
\otimes
\left(
\begin{array}{cc}
  \mu   & 1 \\
  0   & \mu
\end{array}
\right)
\right)=
\left(
\begin{array}{ccc}
  \lambda\mu   & \mu & 0 \\
  \lambda   & 2\lambda\mu+1 & \mu\\
  0   & \lambda & \lambda\mu
\end{array}
\right),
\]
Nontrivial part of equivalence relationship $\backsim$ has form  $(0,0)\backsim (1,1)$. Consequently,
\begin{equation}
 (T^1\ast T^2)_{\overline{i},\overline{j}} 
=\sum\limits_{\overline{m}} T^1_{\overline{i},\overline{m}}T^2_{\overline{m},\overline{j}}+ 
T^1_{\overline{i},(0,0)}T^2_{(1,1),\overline{j}}+
T^1_{\overline{i},(1,1)}T^2_{(0,0),\overline{j}}.
\end{equation}

\end{example}
\begin{example}\label{ex:ex_c}
Let $\mathbb{A}=\prodl_{i=1}^2\mathbb{A}_i$,  $\mathbb{A}_1=\{0,1\}$, $\mathbb{A}_2=\{0,1,2\}$, $F(\overline{i})=\overline{k}\cdot\overline{i}$ with $\overline{k}=(1,1)$.
Then 
\[
\rho_F^{\mathbb{A}}: {\rm Mat}(2)\otimes {\rm Mat}(3)\to {\rm Mat}(4), \rho_F^{\mathbb{A}}(T)=\suml_{i_1,j_1=0}^1\suml_{i_2,j_2=0}^2 T_{i_1,j_1,i_2,j_2} E_{i_1+i_2,j_1+j_2}.
\]
In particular, if 
\[
A=\left(
\begin{array}{cc}
  a_{00}   & a_{01} \\
  a_{10}   & a_{11}
\end{array}
\right),
B=\left(
\begin{array}{ccc}
  b_{00}   & b_{01} & b_{02}\\
  b_{10}   & b_{11} &  b_{12}\\
  b_{20}   & b_{21} &  b_{22}  
\end{array}
\right)
\]
then
\[
\rho_F^{\mathbb{A}}(A\otimes B)=
a_{00}\left(
\begin{array}{cc}
  B   & 0 \\
  0   & 0
\end{array}
\right)
+a_{11}
\left(
\begin{array}{cc}
  0   & 0 \\
  0   & B
\end{array}
\right)
+a_{10}
\left(
\begin{array}{cc}
  0   & 0 \\
  B   & 0
\end{array}
\right)
+a_{01}
\left(
\begin{array}{cc}
  0   & B \\
  0   & 0
\end{array}
\right).
\]
\end{example}
\begin{example}\label{ex:ex_d}
Let $\mathbb{A}=\prodl_{i=1}^2\mathbb{A}_i$,  $\mathbb{A}_1=\mathbb{A}_2=\{0,1\}$,  $F(\overline{i})=\max \overline{i}$. 
Then 
\[
\rho_F^{\mathbb{A}}: {\rm Mat}(2)\otimes {\rm Mat}(2)\to {\rm Mat}(2), \rho_F^{\mathbb{A}}(T)=\suml_{i_1,j_1=0}^1\suml_{i_2,j_2=0}^1 T_{i_1,j_1,i_2,j_2} E_{\max (i_1,i_2),\max (j_1,j_2)}.
\]
In particular, if 
\[
A=\left(
\begin{array}{cc}
  a_{00}   & a_{01} \\
  a_{10}   & a_{11}
\end{array}
\right),
B=\left(
\begin{array}{cc}
  b_{00}   & b_{01} \\
  b_{10}   & b_{11} 
\end{array}
\right)
\]
then
\[
\rho_F^{\mathbb{A}}(A\otimes B)=a_{00}B
+a_{01}
\left(
\begin{array}{cc}
  0   & b_{00}+b_{01} \\
  0   & b_{10}+b_{11}
\end{array}
\right)
+a_{10}
\left(
\begin{array}{cc}
  0   & 0 \\
  b_{00}+b_{10}   & b_{01}+b_{11}
\end{array}
\right)\]
\[+a_{11}
\left(
\begin{array}{cc}
  0   & 0 \\
  0   & b_{00}+b_{10}+b_{01}+b_{11}
\end{array}
\right).
\]
Notice that $\rho_F^{\mathbb{A}}(A\otimes B)=\rho_F^{\mathbb{A}}(B\otimes A)$ and $\rho_F^{\mathbb{A}}$ conserves $\ast$-product of the following form:
\[
(T^1\ast T^2)_{\overline{i},\overline{j}}:= T^1_{\overline{i},(0,0)}T^2_{(0,0),\overline{j}}+(T^1_{\overline{i},(0,1)}+T^1_{\overline{i},(1,0)}+T^1_{\overline{i},(1,1)})(T^2_{(0,1),\overline{j}}+T^2_{(1,0),\overline{j}}+T^2_{(1,1),\overline{j}}).
\]
\end{example}

\begin{center}{\bf Acknowledgments}
\end{center}
 V. F. is supported in part by
the CNPq (402449/2021-5, 302884/2021-1) and by the Fapesp (2018/23690-6). K. Z.  gratefully acknowledge   partial financial supports  from  the NNSF (11871190) and NSERC (311907-2015).

%\begin{rem}
%    Notice that Jordan form of Theorem \ref{thm:JF} is different from standard Jordan form.
%\end{rem}


\begin{thebibliography}{99}

\bibitem{A1934} A. C. Aitken, {\em The normal form of compound and induced matrices}, Proc. London Math. Soc., V. 38, pp. 354-376 (1934).

\bibitem{BM2005} G. Beylkin, M. J. Mohlenkamp, {\em Algorithms for numerical analysis in high dimensions}, SIAM J. Sci. Comput., V. 26, pp. 2133–2159 (2005).

\bibitem{BMW2017} P. Brooksbank, J. Maglione, J. Wilson, {\em A fast isomorphism test for groups whose Lie algebra has genus 2}, J. Algebra, V. 473, pp. 545–590 (2017).

\bibitem{Brualdi1985} R. Brualdi, {\em Combinatorial Verification of the
Elementary Divisors of Tensor Products}, Linear Algebra and its Applications, V.71, pp. 31-47 (1985). 

\bibitem{SBG2004} A. Smilde, R. Bro, P. Geladi, Multi-Way Analysis: Applications in the Chemical Sciences, Wiley, West Sussex, England, 2004.

\bibitem{Comon2001} P. Comon, {\em Tensor decompositions: State of the art and applications}, in Mathematics in Signal Processing V, J. G. McWhirter and I. K. Proudler, eds., Oxford University Press, 2001, pp. 1–24.

\bibitem{FGS2019} V. Futorny, J. Grochow, V. Sergeichuk, 
{\em Wildness for tensors}, Linear Algebra Appl., V. 566, pp. 212–244 (2019).

\bibitem{HK2007} W. Hackbusch, B. N. Khoromskij, {\em Tensor-product approximation to operators and functions in high dimensions}, J. Complexity, V. 23, pp. 697–714 (2007).

\bibitem{HKT2005} W. Hackbusch, B. N. Khoromskij, E. E. Tyrtyshnikov, {\em Hierarchical Kronecker tensor-product approximations}, J. Numer. Math., V.13, pp. 119–156 (2005).

\bibitem{Kolda2001} T. Kolda, {\em Orthogonal tensor decompositions}, SIAM J. Matrix Anal. Appl., V.23, pp. 243–255 (2001). 

\bibitem{Kolda2006} T. Kolda, {\em Multilinear Operators for Higher-Order Decompositions}, Tech. Report SAND2006-2081, Sandia National Laboratories, Albuquerque, NM, Livermore, CA, 2006.

\bibitem{KoldaBader2009} T. Kolda, B. Bader, {\em Tensor decompositions and applications}, SIAM Rev. 51, no. 3, pp. 455–500 (2009).

\bibitem{LMV2000} L. De Lathauwer, B. De Moor, J. Vandewalle,{\em A multilinear singular value decomposition}, SIAM J. Matrix Anal. Appl., V. 21, pp. 1253–1278 (2000). 

\bibitem{LMV2004} L. De Lathauwer, B. De Moor, J. Vandewalle, {\em Computation of the canonical decomposition by means of a simultaneous generalized Schur decomposition}, SIAM J. Matrix Anal. Appl., V. 26, pp. 295–327 (2004). 

\bibitem{LV2004} L. De Lathauwer, J. Vandewalle, {\em Dimensionality reduction in higher-order signal processing and rank-(R1 ,R2 ,.. .,RN ) reduction in multilinear algebra}, Linear Algebra Appl., V. 391, pp. 31–55 (2004).

\bibitem{LZYC2005} N. Liu, B. Zhang, J. Yan, Z. Chen, W. Liu, F. Bai, L. Chien, {\em Text representation: From vector to tensor}, in ICDM 2005: Proceedings of the 5th IEEE International Conference on Data Mining, IEEE Computer Society Press, pp. 725–728 (2005).

\bibitem{LZZ2019}
 M. Liang, B. Zheng, R. Zhao,
{\em Tensor inversion and its application to the tensor equations with Einstein product},
Linear Multilinear Algebra 67, no. 4, pp. 843–870 (2019).

\bibitem{MB2005} D. Muti, S. Bourennane, {\em Multidimensional filtering based on a tensor approach}, Signal Process., V. 85, pp. 2338–2353 (2005).

\bibitem{Roth1934} W. E. Roth, {\em On direct product matrices}, Bull. Amer. Math. Soc., V. 40, pp. 461-468 (1934).

\bibitem{SE2007} B. Savas, L. Eldén, {\em Handwritten digit classification using higher order singular value decomposition}, Pattern Recog., V. 40, pp. 993–1003 (2007).

\bibitem{STF2006} J. Sun, D. Tao, C. Faloutsos, {\em Beyond streams and graphs: Dynamic tensor analysis}, in KDD '06: Proceedings of the 12th ACM SIGKDD International Conference on Knowledge Discovery and Data Mining, ACM Press, pp. 374–383 (2006).

\bibitem{WGE2016} M. Walter, D. Gross, J. Eisert, {Multi-partite entanglement} (extended version of introductory book chapter), available at: arXiv:1612.02437, 2016.

\bibitem{ZG2001} T. Zhang, G. H. Golub, {\em Rank-one approximation to high order tensors}, SIAM J. Matrix Anal. Appl., V. 23, pp. 534–550 (2001).

\end{thebibliography}
\end{document}